\documentclass{amsart}

\usepackage{graphicx}
\usepackage{amssymb, amsmath, amsthm}
\usepackage{geometry}
\usepackage{lipsum}
\usepackage{enumerate}
\usepackage{dsfont}
\usepackage{mathrsfs}
\usepackage{xcolor}

\usepackage{hyperref}

\usepackage[all,cmtip]{xy}
\numberwithin{equation}{section}

\newtheorem{thm}{Theorem}
\newtheorem{coro}{Corollary}
\newtheorem{theorem}{Theorem}[section]
\newtheorem{lemma}[theorem]{Lemma}
\newtheorem{corollary}[theorem]{Corollary}
\newtheorem{proposition}[theorem]{Proposition}

\theoremstyle{remark}

\renewcommand{\a}{\mathbf{a}}
\renewcommand{\b}{\mathbf{b}}

\newcommand{\mc}{\mathcal}

\newcommand{\mb}{\mathbb}
\newcommand{\f}{\mathbf{f}}
\renewcommand{\t}{\mathbf{t}}
\newcommand{\T}{\mathbf{T}}
\newcommand{\x}{\mathbf{x}}
\newcommand{\X}{\mathbf{X}}
\renewcommand{\u}{\mathbf{u}}
\newcommand{\U}{\mathbf{U}}
\newcommand{\ds}{\displaystyle}
\renewcommand{\v}{\mathbf{v}}
\newcommand{\V}{\mathbf{V}}
\newcommand{\h}{\mathbf{h}}
\renewcommand{\k}{\mathbf{k}}
\newcommand{\Y}{\mathbf{Y}}
\newcommand{\Z}{\mathbf{Z}}
\newcommand{\z}{\mathbf{z}}
\newcommand{\s}{{^{_{_{_\leq}}}}}
\newcommand{\SES}{1\rightarrow N\xrightarrow{\iota} G\xrightarrow{\pi} F\rightarrow 1}
\newcommand{\opn}{\operatorname}
\begin{document}
\title[Zeta functions of virtually nilpotent groups]{Zeta functions of virtually nilpotent groups}
\author{Diego Sulca }
\address{Universidad Nacional de C\'ordoba, Ciudad Universitaria, 5000 C\'ordoba.}

\email{sulca@famaf.unc.edu.ar}

\date{\today}
\thanks{The work was partially supported by CONICET and the Research Fund of the KU Leuven}
\begin{abstract}
We prove that the subgroup zeta function and the normal zeta function of a  finitely generated virtually nilpotent group are finite sums
of Euler products of cone integrals over $\mathbb{Q}$ and we deduce from this that they have rational abscissa of convergence and some meromorphic continuation. We also define Mal'cev completions of a finitely generated virtually nilpotent group and we prove that the subgroup growth and the normal subgroup growth of the latter are invariants of its $\mathbb{Q}$-Mal'cev completion.
\end{abstract}

\maketitle

\section*{Introduction}

\noindent For a finitely generated group $G$ and a positive integer $n$, let $a_n^\s(G)$ and $a_n^\lhd(G)$ denote respectively the number of subgroups and normal subgroups of $G$ of index $n$. The subgroup zeta function of $G$ and the local subgroup zeta function of $G$ at a rational prime $p$ are the Dirichlet series
\[\zeta_G^\s(s)=\sum_{n=1}^{\infty}\frac{a_n^\s(G)}{n^s}\ \ \mbox{and}\ \ \zeta_{G,p}^\s(s)=\sum_{k=0}^{\infty}\frac{a_{p^k}^\s(G)}{p^{ks}}.\]
The normal zeta function and the local normal zeta function of $G$ at a prime $p$,
denoted by $\zeta_G^\lhd(s)$ and $\zeta_{G,p}^\lhd(s)$ respectively, are defined similarly using $a_n^\lhd(G)$ instead of $a_n^\s(G)$.
We shall use the symbol $*$ when we address both $\leq$ and $\lhd$ simultaneously.

These series were introduced by Grunewald, Segal and Smith in \cite{GSS} as a tool to study, among other things, the arithmetic properties and the asymptotic behaviour of the sequence $\{a_n^*(G)\}$ in relation with the structure of $G$.
From the general theory of Dirichlet series we know that if for some complex number $s$ the series $\zeta_G^*(s)$ converges then there exists a number $\alpha_G^*\in \{-\infty\}\cup\mathbb{R}$, called the abscissa of convergence of $\zeta_G^*(s)$, such that $\zeta_G^*(s)$ defines an analytic function on the region $\{s\in\mathbb{C}:\Re(s)>\alpha_G^*\}$ and diverges at any $s\in\mathbb{C}$ with $\Re(s)<\alpha_G^*$, and this is the case if and only if the sequence $\{a_n^*(G)\}$ grows at most polynomially, that is, there exists $\alpha>0$ such that $a_n^*(G)\leq n^\alpha$ for all $n$.  In such a case, one says that $G$ has polynomial subgroup growth or polynomial normal subgroup growth. Note that $\alpha_G^\lhd\leq\alpha_G^\s$.

A finitely generated 
residually finite group $G$ (i.\ e.\ the intersection of all its finite index subgroups is trivial) has polynomial subgroup growth if and only if it contains
a finite index subgroup which is solvable of finite rank \cite{LMS}.
The finitely generated virtually nilpotent groups (i.\ e.\ finitely generated groups containing a nilpotent normal subgroup of finite index), and in particular the $\tau$-groups (finitely generated torsion free nilpotent groups) have this property.

For a $\tau$-group $G$ the following results  were obtained in \cite{GSS} and \cite{dSG}.
\begin{enumerate}[I.]
\item $\zeta_{G}^*(s)$ has rational abscissa of convergence $\alpha_G^*\leq h(G)$, where $h(G)$ is the Hirsch length of $G$,
      and this number depends only on the $\mb{Q}$-Mal'cev completion of $G$.
\item $\zeta_{G}^*(s)$ has a meromorphic continuation to the left, that is, there exists $\delta>0$ such that $\zeta_{G}^*(s)$
       can be extended meromorphically  to the region $\{s\in\mb{C}:\Re(s)>\alpha_G^*-\delta\}$. The continued function is holomorphic on the line $\{s\in \mb{C}:\Re(s)=\alpha_G^*\}$ except for a pole  at $s=\alpha_G^*$.
\item $\zeta_{G}^*(s)$ decomposes as an Euler product $\displaystyle{\zeta_{G}^*(s)=\prod_{p\ \textrm{prime}}\zeta_{G,p}^*(s)}$, and for each rational prime $p$,
       there exist polynomials $P_p,Q_p\in\mb{Q}[T]$ of bounded degrees such that
       $\zeta_{G,p}^*(s)=\frac{P_p(p^{-s})}{Q_p(p^{-s})}$.
       \end{enumerate}

The aim of this paper is to prove an extension of these results to the class of finitely generated virtually nilpotent groups. A previous work in this direction is \cite{dSMS}, by du Sautoy, McDermott and Smith, on zeta functions of finitely generated virtually abelian groups, where it is proved that these zeta functions admit a meromorphic continuation to the whole plane. In our case, this property is far from being true in general; even for zeta functions of $\tau$-groups, natural boundaries arise \cite{dSG2}. Rather than following the techniques in \cite{dSMS} we will express our zeta functions in terms of Euler products of cone integrals over $\mb{Q}$ and deduce the desired properties from \cite{dSG}. First, we introduce some notations and recall results about cone integrals in order to state our results precisely.

For a group $G$, let $\mathscr{F}_G^\s$ and  $\mathscr{F}_G^\lhd$ denote respectively the lattice of subgroups and normal subgroups of $G$ of finite index.  

Let $S:1\rightarrow N\xrightarrow{\iota} G\xrightarrow{\pi} F\rightarrow 1$ be a short exact sequence of groups with $F$ finite. For $K\in\mathscr{F}_F^*$ and a rational prime $p$ we define
\begin{align*}
\mathscr{F}_{S,K}^*&=\{A\in\mathscr{F}_G^*:\pi(A)=K\}\\
\mathscr{F}_{S,K,p}^*&=\{A\in\mathscr{F}_G^*:\pi(A)=K,[\pi^{-1}(K):A]\ \mbox{is a power of }p\}.\end{align*}
Since $\ds\mathscr{F}_G^*=\bigcup_{K\in\mathscr{F}_F^*}\mathscr{F}_{S,K}^*$ and $[G:A]=[F:K][\pi^{-1}(K):A]$ for any $A\in \mathscr{F}_{S,K}^*$, one obtains
\begin{align}\label{descomposition into sum of relative zeta functions}
\zeta_G^*(s)&=\sum_{A\in \mathscr{F}_G^*}[G:A]^{-s}
 =\sum_{K\in\mathscr{F}_{F}^*}[F:K]^{-s}\sum_{A\in\mathscr{F}_{S,K}^*}[\pi^{-1}(K):A]^{-s}.
\end{align}
Defining 
\begin{align}\label{relative zeta function}\zeta_{S,K}^*(s)=\sum_{A\in\mathscr{F}_{S,K}^*}[\pi^{-1}(K):A]^{-s},
\end{align}
the last expression becomes
\begin{align}\label{decomposition into sum of relative zeta function with new notation}
  \zeta_G^*(s)=\sum_{K\in\mathscr{F}_{F}^*}[F:K]^{-s}\zeta_{S,K}^*(s).
\end{align}
It follows that the analytic properties of $\zeta_G^*(s)$ can be deduced from the analytic properties of the series $\zeta_{S,K}^*(s)$ which, therefore, will be central in our study. For a rational  prime $p$, we shall also consider the series
\begin{align}\label{local relative zeta function}\zeta_{S,K,p}^*(s)=\sum_{A\in\mathscr{F}_{S,K,p}^*}[\pi^{-1}(K):A]^{-s}.\end{align}

These series are also the focus of the aforementioned work \cite{dSMS} on zeta functions of finitely generated virtually abelian groups. Assuming that $N$ is a finitely generated free abelian group, then there is an Euler product decomposition \cite[Proposition 2.2]{dSMS}
\begin{align}\label{Euler product for relative zeta functions}\zeta_{S,K}^*(s)=\prod_{p\  \textrm{prime}}\zeta_{S,K,p}^*(s).
\end{align}
This will be proved again in Section 2 by observing that we only need to assume that $N$ is a $\tau$-group.

A finitely generated virtually nilpotent group is also a virtually $\tau$-group. Thus, rather than talking about finitely generated virtually nilpotent groups, we shall consider short exact sequences of groups $S:1\rightarrow N\xrightarrow{\iota} G\xrightarrow{\pi} F\rightarrow 1$, where $N$ is a $\tau$-group and $F$ is finite and, by abuse of the language, we shall call them virtually $\tau$-groups too.

Let $S:1\rightarrow N\xrightarrow{\iota} G\xrightarrow{\pi} F\rightarrow 1$ be a virtually $\tau$-group. For $K\in \mathscr{F}_F^*$, several analytic properties of $\zeta_{S,K}^*(s)$ can be obtained by expressing it as an Euler product of cone integrals over $\mb{Q}$.
Given a natural number $m$, we call a collection of polynomials
$ \mc{D}=\{f_0,g_0;f_1,g_1,\ldots,f_l,g_l \}$,
with $f_0,g_0,f_1,g_1,\ldots,f_l,g_l\in\mb{Q}[x_1,\ldots,x_m]$
a  \emph{cone integral data}. For each rational prime $p$, we associate to $\mc{D}$ the closed subset of $\mb{Z}_p^m$
\[ \mc{M}(\mc{D},p)=\{\x\in\mb{Z}_p^m:v_p(f_i(\x))\leq v_p(g_i(\x))\ \mbox{for}\ i=1,\ldots,l\}, \]
where $v_p$ is the $p$-adic valuation on $\mb{Z}_p$.
Then the $p$-adic integral
\[ Z_{\mc{D}}(s,p)=\int_{\mc{M}(\mc{D},p)}|f_0(\x)|_p^s|g_0(\x)|_pd\mu(\x), \]
where $\mu$ is the normalized Haar measure on $\mb{Z}_p^m$ and $s$ is a complex number,
is called a \emph{cone integral} over $\mb{Q}$.
Each $Z_{\mc{D}}(s,p)$ is a power series
\[ Z_{\mc{D}}(s,p)=\sum_{i=0}^{\infty}a_{p,i}(\mc{D})p^{-is} \]
with rational coefficients and it is a rational function in $p^{-s}$ ~\cite{De}.
A complex function $Z(s)$ is said to be an Euler product of cone integrals over $\mb{Q}$ with cone integral data $\mc{D}$ if
\[Z(s)=\prod_{\substack{p\ \textrm{prime}\\ a_{p,0}(\mc{D})\neq 0}}(a_{p,0}^{-1}(\mc{D})\cdot Z_{\mc{D}}(s,p)),\]
and in this case one writes $Z(s)=Z_\mc{D}(s)$.
It is proved in \cite{dSG} that such a function $Z(s)$ is expressible as a Dirichlet series $\sum_{n=1}^\infty a_nn^{-s}$
with non-negative coefficients,
 and if  the function $Z_\mc{D}(s,p)$ is not the constant function for almost all primes $p$, then the following hold:
\begin{enumerate}[I.]
\item {The abscissa of convergence $\alpha_{\mc{D}}$ of $Z_{\mc{D}}(s)=\sum_{n=1}^{\infty}a_nn^{-s}$ is a rational number.}
\item {$Z_{\mc{D}}(s)$ has meromorphic continuation to $\{s\in\mb{C}
:\Re(s)>\alpha_\mc{D}-\delta\}$ for some $\delta >0$. The continued function is holomorphic on the line  $\{s\in\mb{C}:\Re(s)=\alpha_{\mc{D}}\}$ except for a pole at $s=\alpha_{\mc{D}}$.}
\item {The abscissa of convergence of each local factor is strictly less than  $\alpha_{\mc{D}}$.}
\item {Each $Z_\mc{D}(s,p)$ is a rational function $\frac{P_p(p^{-s})}{Q_p(p^{-s})}$, where $P_p$ and $Q_p$ are polynomials with rational coefficients of degrees bounded by a constant not depending of $p$.}
\end{enumerate}

Our first main result is
\begin{thm}
Let $S:\SES$ be a virtually $\tau$-group and let $*\in\{\leq,\lhd\}$.
Then, for each $K\in\mathscr{F}_F^*$, $\zeta_{S,K}^*(s)$ is an Euler product of cone integrals over $\mb{Q}$. More precisely, for some cone integral data $\mc{D}^*$ (depending on $K$) we have
\[\zeta_{S,K}^{*}(s)=Z_{\mc{D}^*}(s-h(N)-|K|+1).\]
\end{thm}
As a consequence of one of the main steps in the proof of Theorem 1 we get
\begin{coro}
The subgroup zeta function and the normal zeta function of a $\tau$-group are Euler products of cone integrals over $\mb{Q}$.
\end{coro}
This corollary allows us to apply the properties of cone integrals to zeta functions of $\tau$-groups directly, without the necessity of linearising  the computation of cone integral by using the Mal'cev correspondence between $\tau$-groups and Lie rings and with the lack of some local factors \cite[Remark after Corollary 5.6]{dSG}. However, computations of cone integrals here are much more difficult. The method to express $\zeta_{S,K}^*(s)$ as an Euler product of cone integrals is used in \cite{S} to compute the zeta functions of all finitely generated torsion free virtually nilpotent groups of Hirsch length 3, called 3-dimensional almost Bieberbach groups.

For a binomial domain $R$ (e.\ g.\ a field of characteristic zero), the $R$-Mal'cev completion of a virtually $\tau$-group (existence and uniqueness) will be considered in Section 1. 
If $S:\SES$ is a virtually $\tau$-group, the $R$-Mal'cev completion of $S$ is a short exact sequence $\mathfrak{S}:1\rightarrow \mathfrak{N}\xrightarrow{\iota_\mathfrak{S}} \mathfrak{G}\xrightarrow{\pi_\mathfrak{S}} \mathfrak{F}\rightarrow 1$ with a morphism $(i,j,k):S\rightarrow \mathfrak{S}$ of short exact sequences of groups such that $i:N\rightarrow \mathfrak{N}$ is an $R$-Mal'cev completion of $N$ and $k:F\rightarrow \mathfrak{F}$ is an isomorphism.  As a consequence of Theorem 1 we shall obtain
\begin{thm}
For $i=1,2$, let $S_i:1\rightarrow N_i\rightarrow G_i\rightarrow F_i\rightarrow 1$  be  a $\tau$-group with $\mb{Q}$-Mal'cev completion $\mathfrak{S}_i$. If $\mathfrak{S}_1$ and $\mathfrak{S}_2$ are isomorphic (as short exact sequences of groups), then there exists an isomorphism $\gamma:F_1\rightarrow F_2$ such that for $*\in\{\leq,\lhd\}$ the series $\zeta_{S_1,K}^*(s)$ and $\zeta_{S_2,\gamma(K)}^*(s)$ have the same abscissa of convergence for all $K\in\mathscr{F}_{F_1}^*$.
\end{thm}

Finally we collect the results stated above and some others on bounds for the subgroup growth and the normal subgroup growth of virtually nilpotent groups (\cite[Proposition 5.6.4]{LS} and Proposition \ref{normal subgroup growth}), and we include them in a last theorem.

\begin{thm}  Let $S:\SES$ be a virtually $\tau$-group. For $*\in\{\leq,\lhd\}$ the following hold.
\begin{enumerate}[I.]
\item $\zeta_{G}^{*}(s)$ has rational abscissa of convergence $\alpha_G^{*}\leq \alpha^*_N+1\leq h(G)+1$, which depends only on the $\mb{Q}$-Mal'cev completion of $S$.
\item There exists $\delta>0$ such that $\zeta_{G}^*(s)$ has a meromorphic continuation to the region $\Re(s)>\alpha^*_G-\delta$,
 and the line $\Re(s)=\alpha^*_G$ contains just one pole of $\zeta_{G}^*(s)$  (at the point $s=\alpha^*_G$).
\item More precisely, $\zeta_{G}^*(s)=\sum_{K\in\mathscr{F}_{F}^*}[F:K]^{-s}\zeta_{S,K}^*(s)$ and for each $K\in\mathscr{F}_F^*$ the series $\zeta_{S,K}^*(s)$ has rational abscissa of convergence $\alpha^*_{S,K}$, it has meromorphic continuation to  $\{s\in\mb{C}:\Re(s)>\alpha^*_{S,K}-\delta\}$ for some $\delta>0$ with a unique pole on the line $\{s\in\mb{C}:\Re(s)=\alpha^*_{S,K}\}$ (at the point $s=\alpha^*_{S,K}$), it decomposes as an Euler product $\zeta_{S,K}^*(s)=\prod_{p\ \textrm{prime}} \zeta_{S,K,p}^*(s)$ and each $\zeta_{S,K,p}^*(s)$ is a quotient of polynomials in $p^{-s}$ with rational coefficients of bounded degrees (i.e.\  bounded by a number not depending on $p$). The multiset  $\{\alpha^*_{S,K}:K\in \mathscr{F}_F^*\}$ depends only on the $\mb{Q}$-Mal'cev completion of $S$.
\end{enumerate}
\end{thm}

\subsection*{Acknowledgment} The author wishes to thank the referee for many useful
suggestions that helped to improve the exposition of this paper.

\section{Mal'cev completions of virtually nilpotent groups}
\noindent In this section we define the $R$-Mal'cev completion of a virtually $\tau$-group for any binomial domain $R$. First we give some preliminaries about $R$-Mal'cev completions of $\tau$-groups and some others on group extensions. Then we prove existence and uniqueness (up to isomorphism) of the $R$-Mal'cev completion of a virtually $\tau$-group. 

\subsection{Preliminaries.}

Let $R$ be a binomial domain, that is, an integral domain of characteristic zero such that, for any $r\in R$ and any $k\in\mb{N}$, the binomial coefficient $\binom{r}{k}=\frac{r(r-1)\ldots (r-k+1)}{k!}$ lies in $R$. Examples are the ring of integers $\mb{Z}$, the ring of $p$-adic numbers $\mb{Z}_p$ for any rational prime $p$, and any field of characteristic zero. A nilpotent $R$-powered group  is a nilpotent group $\mathfrak{N}$ where for all $n\in \mathfrak{N}$ and $r\in R$  an element $n^r\in \mathfrak{N}$ has been defined such that the following hold:
\begin{itemize}
\item[(i)] $n^1=n$; $n^{r_1+r_2}=n^{r_1}n^{r_2}$; $n^{r_1r_2}=(n^{r_1})^{r_2}$ for all $n\in\mathfrak{N}$, $r_1,r_2\in R$.
\item[(ii)] $m^{-1}n^rm=(m^{-1}nm)^r$ for all $m,n\in \mathfrak{N}$, $r\in R$.
\item[(iii)] The Hall-Petresco formula holds for all $k$-tuples $(n_1,\ldots,n_k)$ in $\mathfrak{N}$ and $r\in R$ \cite[Chap.\ 6]{W}.
\end{itemize}
A morphism $\varphi:\mathfrak{M}\rightarrow \mathfrak{N}$ between nilpotent $R$-powered groups is a group homomorphism such that $\varphi(m^r)=\varphi(m)^r$ for all $m\in \mathfrak{M}$, $r\in R$ and it will be called an $R$-morphism. More about these groups can be found in \cite[Chap.\ 10]{W}.

If $N$ is a $\tau$-group with Mal'cev basis $\{x_1,\ldots,x_h\}$ then any element $x\in N$ can be written uniquely in the form $x=x_1^{a_1}\ldots x_h^{a_h}$ with $a_1,\ldots,a_h\in\mb{Z}$. By \cite[Theorem 6.5]{H}, for each $i=1,\ldots,h$, there exist functions $f_i(\a,\b)=f_i(a_1,\ldots,a_h,b_1,\ldots,b_h)$ and $g_i(\a,r)=g_i(a_1,\ldots,a_h,r)$, which are polynomials in $\binom{a_1}{k_1},\ldots,\binom{a_h}{k_h},\binom{b_1}{l_1},\ldots,\binom{b_h}{l_h},r$, with $k_j,l_j\in\mb{N}$, with integer coefficients and such that
\begin{align}\label{multiplicacion en t-grupos}
x_1^{a_1}\ldots x_h^{a_h}x_1^{b_1}\ldots x_h^{b_h}&=x_1^{f_1(\a,\b)}\ldots x_h^{f_h(\a,\b)}\\
\label{exponenciacion en t-grupos} (x_1^{a_1}\ldots x_h^{a_h})^r&=x_1^{g_1(\a,r)}\ldots x_h^{g_h(\a,r)}.
\end{align}
Following \cite[Page 48]{H}, a nilpotent $R$-powered group $N^R$ can be obtained from $N$ as follows: the elements of $N^R$ are all the formal products $x_1^{r_1}\ldots x_h^{r_h}$ with $r_1,\ldots,r_h\in R$, and the multiplication and exponentiation in $N^R$ are defined by means of the polynomials $f_i(\a,\b)$ and $g_i(\a,r)$ as in (\ref{multiplicacion en t-grupos}) and (\ref{exponenciacion en t-grupos}). Note that this makes sense because $R$ is a binomial domain.
  
The group $N^R$ and the obvious injection $i_N:N\rightarrow  N^R$ satisfy the following universal property: If $\mathfrak{M}$ is a nilpotent $R$-powered group and $f:N\rightarrow \mathfrak{M}$ is a morphism of groups, then there exists a unique $R$-morphism  $\tilde{f}:N^R\rightarrow \mathfrak{M}$ such that $\tilde{f}i_N=f$ \cite[Theorem 10.14]{W}. In particular $N\mapsto N^R$ defines a functor from the category of $\tau$-groups to the category of nilpotent $R$-powered groups: if $f:M\rightarrow N$ is a morphism of $\tau$-groups, $f^R:M^R\rightarrow N^R$ is the morphism of nilpotent $R$-powered groups such that $f^Ri_N=i_Mf$. We call $f^R$ the $R$-extension of $f$. Finally, any pair $(i,\mathfrak{N})$ satisfying the  universal property of $(i_N,N^R)$ will be called an \emph{$R$-Mal'cev completion of $N$}; of course, they are all canonically isomorphic.

Now we recall some facts about group extensions. Let $S:\SES$ be a short exact sequence of groups. Any set $T=\{g_f\in G:f\in F, g_1=1, \pi(g_f)=f\}\subseteq G$ defines a pair of maps $\sigma: F\rightarrow \opn{Aut}(N)$ and $\psi:F\times F\rightarrow N$, where $\sigma(f)\in \opn{Aut}(N)$ is the unique automorphism such that $\iota(\sigma(f)(n))=g_f\iota(n)g_f^{-1}$, and $\psi(f,f')$ is the unique element of $N$ such that $g_fg_{f'}=\iota(\psi(f,f'))g_{ff'}$. Note that multiplication in $G$ is given by $(\iota(n)g_f)(\iota(n')g_{f'})=\iota(n\sigma(f)(n')\psi(f,f'))g_{ff'}$. The pair $(\sigma,\psi)$ will be called the cocycle associated to $T$ because it satisfies the following \emph{cocycle conditions}.
\begin{enumerate}\label{cocycle conditions}
\item $\sigma(f)\sigma(f')=\mu(\psi(f,f'))\sigma(ff')$ $\forall f,f'\in F$ ,
\item  $\psi(f,f')\psi(ff',f'')=\sigma(f)(\psi(f,f'))\psi(f,f'f'')$ $\forall f,f',f''\in F$,
\end{enumerate}
where for $n\in N$, $\mu(n)\in \opn{Aut}(N)$ denotes conjugation by $n$. The set $N\times F$ can be given the structure of a group by defining $(n,f)*_{(\sigma,\psi)} (n',f')=(n\sigma(f)(n')\psi(f,f'),ff')$, and the map $j:G\rightarrow N\times F$ sending $\iota(n)g_f$ to $(n,f)$ is an isomorphism. Moreover, we have an exact sequence $S_{(\sigma,\psi)}:1\rightarrow N\rightarrow N\times F\rightarrow F\rightarrow 1$, where the first map is $n\mapsto (n,1)$ and the second map is the projection $(n,f)\mapsto f$. Clearly $(id_N,j,id_F):S\rightarrow S_{(\sigma,\psi)}$ is an isomorphism of short exact sequences. 

Finally, given groups $N$, $F$ and a couple of maps  $\sigma: F\rightarrow \opn{Aut}(N)$ and $\psi:F\times F\rightarrow N$ such that $\sigma(1)=id_N$ and $\psi(f,1)=\psi(1,f)=1$ for all $f\in F$, if $(\sigma,\psi)$ satisfies the cocycle conditions stated above then we can define a group structure on $N\times F$ and a short exact sequence $S_{(\sigma,\psi)}$ just as we did in the last paragraph. We denote the group $N\times F$ with this operation by $N\times_{(\sigma,\psi)}F$.

\subsection{The $R$-Mal'cev completion of a virtually $\tau$-group.} We keep the notation of the last subsection. By a virtually nilpotent $R$-powered group we mean a short exact sequence of groups $\mathfrak{S}:1\rightarrow \mathfrak{N}\rightarrow \mathfrak{G}\rightarrow\mathfrak{F}\rightarrow 1$ where
\begin{enumerate}[a)]
\item  $\mathfrak{N}$ is a nilpotent $R$-powered group,
\item $\mathfrak{F}$ is a finite group, and
\item  for any $g\in\mathfrak{G}$ the automorphism of $\mathfrak{N}$ induced by conjugation by $g$ is an $R$-morphism.
\end{enumerate}
 The morphisms of virtually nilpotent $R$-powered groups, also called $R$-morphisms, are the morphisms of short exact sequences $(\alpha,\beta,\gamma):\mathfrak{S}\rightarrow\mathfrak{S}'$ such that $\alpha$ is an $R$-morphism. An \emph{$R$-Mal'cev completion of a virtually $\tau$-group} $S:\SES$ is a virtually nilpotent $R$-powered group $\mathfrak{S}:1\rightarrow \mathfrak{N}\xrightarrow{\iota_\mathfrak{S}} \mathfrak{G}\xrightarrow{\pi_\mathfrak{S}} \mathfrak{F}\rightarrow 1$ with a morphism $(i,j,k):S\rightarrow \mathfrak{S}$ of short exact sequences such that $i:N\rightarrow \mathfrak{N}$ is an $R$-Mal'cev completion of $N$ and $k:F\rightarrow \mathfrak{F}$ is an isomorphism.
Note that $j$ must be injective and $\mathfrak{G}=\iota_\mathfrak{S}(\mathfrak{N})j(G)$.
\begin{proposition}\label{Existence of the Malcev completion} An $R$-Mal'cev completion of a virtually $\tau$-group always exists. 
\end{proposition}
\begin{proof}
Given a virtually $\tau$-group $S:\SES$, consider the $R$-Mal'cev completion $i_N:N\rightarrow N^R$ of $N$. Choose a set $T=\{g_f\in G:f\in F, g_1=1, \pi(g_f)=f\}\subseteq G$ and let $(\sigma,\psi)$ be the associated cocycle. For $f\in F$, we put $\tilde{\sigma}(f)=(\sigma(f))^R$ ($R$-extension of $\sigma(f)$) and $\tilde{\psi}=i_N\psi$. It is easy to see that $(\tilde{\sigma},\tilde{\psi})$ satisfies the cocycle conditions. In fact, the second cocycle condition holds because it already holds for the pair $(\sigma,\psi)$, and to check the first cocycle condition for given $f,f'\in F$, we note that $\tilde{\sigma}(f)\tilde{\sigma}(f')$ and $\mu(\tilde{\psi}(f,f'))\tilde{\sigma}(ff')$ are $R$-extensions of $\sigma(f)\sigma(f')=\mu(\psi(f,f'))\sigma(ff')$ and therefore they are equal.

We obtain a group $G_T^R=N^R\times_{(\tilde{\sigma},\tilde{\psi})}F$ and a natural short exact sequence $1\rightarrow N^R\rightarrow G^R_T\rightarrow F\rightarrow 1$. Defining $j:G\rightarrow G^R_T$ by $j(\iota(n)g_f)=(i_N(n),f)$  it is easy to check that $(i_N,j,id_F)$ is an $R$-Mal'cev completion of $S$.
\end{proof}

Our next task is to prove uniqueness of the $R$-Mal'cev completion up to a canonical isomorphism by showing that any $R$-Mal'cev completion solves a universal problem.
Suppose we are given a virtually nilpotent group $\SES$ with an $R$-Mal'cev completion $(i,j,k):S\rightarrow \mathfrak{S}$, where $\mathfrak{S}:1\rightarrow \mathfrak{N}\xrightarrow{\iota_\mathfrak{S}} \mathfrak{G}\xrightarrow{\pi_\mathfrak{S}} \mathfrak{F}\rightarrow 1$. Choose a subset $T=\{g_f\in G:f\in F,\pi(f)=f, g_1=1\}\subseteq G$, let $(\sigma,\psi)$ be the associated cocycle and let $(\tilde{\sigma},\tilde{\psi})$ be defined as in the proof of Proposition \ref{Existence of the Malcev completion} (replacing $N^R$ by $\mathfrak{N}$). It is easy to see that $(\sigma_\mathfrak{S},\psi_\mathfrak{S})=(\tilde{\sigma} k^{-1},\tilde{\psi}(k^{-1},k^{-1}))$ is the cocycle of $\mathfrak{S}$ associated to $T_\mathfrak{S}=j(T)$, and clearly $\sigma_\mathfrak{S}(\pi_\mathfrak{S}(j(g_f)))=\tilde{\sigma}(f)$ and $\psi_\mathfrak{S}(\pi_\mathfrak{S}(j(g_f)),\pi_\mathfrak{S}(j(g_{f'})))=\tilde{\psi}(f,f')$ for all $f,f'\in F$. We shall make use of these facts without mention in the proof of the next proposition.
\begin{proposition}\label{functoriality property of the Malcev completion}
Let $(i,j,k):S\rightarrow\mathfrak{S}$ be an $R$-Mal'cev completion of the virtually $\tau$-group $S$, where $S:\SES$ and $\mathfrak{S}:1\rightarrow \mathfrak{N}\xrightarrow{\iota_\mathfrak{S}} \mathfrak{G}\xrightarrow{\pi_\mathfrak{S}} \mathfrak{F}\rightarrow 1$. If $\mathfrak{T}:1\rightarrow \mathfrak{M}\xrightarrow{\iota_\mathfrak{T}}\mathfrak{H}\xrightarrow{\pi_\mathfrak{T}}\mathfrak{K}\rightarrow 1$ is another virtually nilpotent $R$-powered group and $(\alpha,\beta,\gamma):S\rightarrow\mathfrak{T}$ is a morphism of short exact sequences of groups, then there exists a unique morphism of virtually nilpotent $R$-powered groups $(\tilde{\alpha},\tilde{\beta},\tilde{\gamma}):\mathfrak{S}\rightarrow\mathfrak{T}$ such that $(\alpha,\beta,\gamma)=(\tilde{\alpha},\tilde{\beta},\tilde{\gamma})\circ (i,j,k)$.
\end{proposition}
\begin{proof}
Since $i:N\rightarrow\mathfrak{N}$ is an $R$-Mal'cev completion of $N$ there exists a unique morphism of nilpotent $R$-powered groups $\tilde{\alpha}:\mathfrak{N}\rightarrow\mathfrak{M}$ such that $\tilde{\alpha} i=\alpha$; since $k$ is an isomorphism of groups there exists a unique morphism of groups $\tilde{\gamma}:\mathfrak{F}\rightarrow\mathfrak{K}$ such that $\tilde{\gamma}k=\gamma$.  Choose a subset $T=\{g_f\in G:f\in F_1,\pi(g_f)=f, g_1=1\}\subseteq G$, let $(\sigma,\psi)$ be the associated cocycle and let $(\tilde{\sigma},\tilde{\psi})$ and $(\sigma_\mathfrak{S},\psi_\mathfrak{S})$ be defined as in the paragraph before the proposition.
 Since $\mathfrak{G}=\iota_\mathfrak{S}(\mathfrak{N})j(G)$,  to define $\tilde{\beta}$ in order to satisfy the conditions $\tilde{\beta} j=\beta$ and $\tilde{\beta}\iota_\mathfrak{S}=\iota_\mathfrak{T}\tilde{\alpha}$ we are forced to do it as 
\begin{align*}
\tilde{\beta}(\iota_\mathfrak{S}(n)j(g_f))=\iota_\mathfrak{T}(\tilde{\alpha}(n))\beta(g_f).
\end{align*}
Let us prove that $\tilde{\beta}$, defined in this way, is a morphism of groups. On the one hand
\begin{align}
&\tilde{\beta}(\iota_\mathfrak{S}(n)j(g_f)\iota_\mathfrak{S}(n')j(g_{f'}))\nonumber \\ 
&=\tilde{\beta}(\iota_\mathfrak{S}(n\sigma_\mathfrak{S}(\pi_\mathfrak{S}(j(g_f)))(n')\psi_\mathfrak{S}(\pi_\mathfrak{S}(j(g_f)),\pi_\mathfrak{S}(j(g_{f'}))))j(g_{ff'}))\nonumber \\
&=\tilde{\beta}(\iota_\mathfrak{S}(n\tilde{\sigma}(f)(n')\tilde{\psi}(f,f'))j(g_{ff'}))\nonumber \\
&=\iota_\mathfrak{T}(\tilde{\alpha}(n\tilde{\sigma}(f)(n')\tilde{\psi}(f,f')))\beta(g_{ff'})\nonumber \\
&=\iota_\mathfrak{T}(\tilde{\alpha}(n)\tilde{\alpha}(\tilde{\sigma}(f)(n'))\tilde{\alpha}(\tilde{\psi}(f,f')))\beta(g_{ff'}).\label{LHS}
\end{align}
On the other hand  
\begin{align*}
&\tilde{\beta}(\iota_\mathfrak{S}(n)j(g_f))\tilde{\beta}(\iota_\mathfrak{S}(n')j(g_{f'}))\\
&=\iota_\mathfrak{T}(\tilde{\alpha}(n))\beta(g_f) \iota_\mathfrak{T}(\tilde{\alpha}(n'))\beta(g_{f'})\\
&=\iota_\mathfrak{T}(\tilde{\alpha}(n)\tau_{\beta(g_f)}(\tilde{\alpha}(n')))\beta(g_f)\beta(g_{f'}),
\end{align*}
where $\tau_{\beta(g_f)}\in \opn{Aut}_R(\mathfrak{M})$ is the automorphism such that $\iota_\mathfrak{T}(\tau_{\beta(g_f)}(m))=\beta(g_f)\iota_\mathfrak{T}(m)(\beta(g_f))^{-1}$; since $\beta(g_f)\beta(g_{f'})=\beta(g_{f}g_{f'})=\beta(\iota(\psi(f,f'))g_{ff'})
=\iota_\mathfrak{T}(\tilde{\alpha}(i(\psi(f,f'))))\beta(g_{ff'})$, our last expression becomes
\begin{align}\label{RHS}
\iota_\mathfrak{T}(\tilde{\alpha}(n)\tau_{\beta(g_f)}(\tilde{\alpha}(n'))\tilde{\alpha}(\tilde{\psi}(f,f')))\beta(g_{ff'}).   
\end{align}
Looking at (\ref{LHS}) and (\ref{RHS}) we see that we only need to check that the morphisms $\iota_\mathfrak{T}\tilde{\alpha}\tilde{\sigma}(f)$ and $\iota_\mathfrak{T}\tau_{\beta(g_f)}\tilde{\alpha}$  from $\mathfrak{N}$ to $\iota_\mathfrak{T}(\mathfrak{M})$ are equal. Since they are clearly $R$-morphisms  it is enough to check that they are both extension to $\mathfrak{N}$  of the map $N\rightarrow \iota_\mathfrak{R}(\mathfrak{M})$ given by $n\mapsto \beta(g_f\iota(n)g_f^{-1})$. 
For the first map we have
\begin{align*}
\iota_\mathfrak{T}(\tilde{\alpha}(\tilde{\sigma}(f)(i(n))))&=\iota_\mathfrak{T}(\tilde{\alpha}(i(\sigma(f)(n))))=\iota_\mathfrak{T}(\alpha(\sigma(f)(n)))\\
&=\beta(\iota(\sigma(f)(n)))=\beta(g_f\iota(n)g_f^{-1}).
\end{align*}
For the second map we have
\begin{align*}
\iota_\mathfrak{T}(\tau_{\beta(g_f)}(\tilde{\alpha}(i(n))))&=\beta(g_f)\iota_\mathfrak{T}(\tilde{\alpha}(i(n)))\beta(g_f^{-1})\\
&=\beta(g_f)\iota_\mathfrak{T}(\alpha(n)) \beta(g_f^{-1})\\
&=\beta(g_f)\beta(\iota(n))\beta(g_f^{-1})\\
&=\beta(g_f\iota_1(n)g_f^{-1}).
\end{align*}
We conclude that $\tilde{\beta}$ is a morphism and by definition it is the unique morphism satisfying $\tilde{\beta}\iota_\mathfrak{S}=\iota_\mathfrak{T}\tilde{\alpha}$ and $\tilde{\beta}j=\beta$. We still have to verify that $\pi_\mathfrak{T}\tilde{\beta}=\tilde{\gamma}\pi_\mathfrak{S}$. This follows from
\begin{align*}
\pi_\mathfrak{T}(\tilde{\beta}(\iota_\mathfrak{S}(n)j(g_f)))&=\pi_\mathfrak{T}(\iota_\mathfrak{T}(\tilde{\alpha}(n))\beta(g_f))=
\pi_\mathfrak{T}(\beta(g_f))=\gamma(f)=\tilde{\gamma}(k(f))\\
&=\tilde{\gamma}(\pi_\mathfrak{S}(j(g_f)))=\tilde{\gamma}\pi_\mathfrak{S}(\iota(n)j(g_f)), \end{align*}
\end{proof}
\begin{corollary}\label{uniqueness of the Malcev completion} If $(i_1,j_1,k_1):S\rightarrow \mathfrak{S}_1$ and $(i_2,j_2,k_2):S\rightarrow \mathfrak{S}_2$ are $R$-Mal'cev completions of a virtually $\tau$-group $S$, then there exists a unique $R$-isomorphism $(u,v,w):\mathfrak{S}_1\rightarrow \mathfrak{S}_2$ of such that $(u i_1,v j_1,w k_1)=(i_2,j_2,k_2)$.
\end{corollary}
\begin{proof}
This follows from Proposition \ref{functoriality property of the Malcev completion} with $\alpha=i_2,\beta=j_2$ and $\gamma=i_3$.
\end{proof}
This allows us to talk about \emph{the} $R$-Mal'cev completion of a virtually $\tau$-group. Moreover, as a formal consequence of Proposition \ref{functoriality property of the Malcev completion} we also obtain
\begin{corollary}\label{functoriallity property of the Malcev completion, corollary} The $R$-Mal'cev completion defines a functor from the category of virtually $\tau$-groups to the category of virtually nilpotent $R$-powered groups.
\end{corollary}

\section{Expression of zeta functions as Euler products of cone integrals}
\noindent The main task of this section is the proof of Theorem 1. First, we will reformulate this theorem (see Theorem \ref{main theorem}) and  after several steps we will prove the new version. As a result of one of these steps, we will obtain Corollary 1 (see Corollary \ref{Algoritmo,corollary 2}). Then, after establishing some consequences of Theorem 1, we will make use of some facts about Mal'cev completions of virtually $\tau$-groups that we have proved in the last section to give a proof of Theorem 2. At the end of the section we shall give the proof of Theorem 3.  
\subsection{Decomposition of $\zeta_{S,K}^*(s)$ as an Euler product}
Let $S:\SES$ be a virtually $\tau$-group and let $\widehat{\mathbb{Z}}$ be the profinite completion of $\mb{Z}$, which is a binomial domain. The $\widehat{\mb{Z}}$-Mal'cev completion of $S$ is just $(i,j,id_F):S\rightarrow \widehat{S} \ (=1\rightarrow\widehat{N}\xrightarrow{\widehat{\iota}}\widehat{G}\xrightarrow{\widehat{\pi}}\widehat{F}\rightarrow 1)$, where $i:N\rightarrow \widehat{N}$, $j:G\rightarrow\widehat{G}$ and $id_F:F\rightarrow \widehat{F}$ are the profinite completions of $N$, $G$ and $F$ respectively. We shall refer to $S\rightarrow \widehat{S}$ as the profinite completion of $S$. For simplicity, we assume that $N,G$ and $\widehat{N}$ are all subgroups of $\widehat{G}$ and $i,j,\iota$ and $\widehat{\iota}$ are inclusion maps. This will allow us, for example, to write intersections of subgroups of $G$ with $N$ instead of writing pre-images by the map $\iota$.

Every finite index subgroup of $\widehat{G}$ is open because $G$ is finitely generated \cite{NS}, therefore the usual correspondence between finite index subgroups of a group and open subgroups of its profinite completion ($A\mapsto \overline {A}$, where $\overline{A}$ is the topological closure of $A$ in $\widehat{G}$ ) gives an isomorphism of lattices $\mathscr{F}_G^*\rightarrow\mathscr{F}_{\widehat{G}}^*$ preserving the relative index ($[\overline{B}:\overline{A}]=[B:A]$ if $A\leq B$), and such that the composition $\mathscr{F}_G^*\rightarrow\mathscr{F}_{\widehat{G}}^*\xrightarrow{\widehat{\pi}}
\mathscr{F}_F^*$ coincides with the map $\mathscr{F}_{G}^*\xrightarrow{\pi}\mathscr{F}_F^*$. These facts imply that for $K\in\mathscr{F}_F^*$, $A\in\mathscr{F}_{S,K}^*$ if and only if $\overline{A}\in\mathscr{F}_{\widehat{S},K}^*$ and, for each prime $p$, $A\in\mathscr{F}_{S,K,p}^*$ if and only if $\overline{A}\in\mathscr{F}_{\widehat{S},K,p}^*$. We conclude
\begin{proposition}\label{equality between zeta function of a group and that of its profinite completion}
Let $S$ be a virtually $\tau$-group with profinite completion $\widehat{S}$. Then for each $K\in\mathscr{F}_F^*$ we have $\zeta_{S,K}^*(s)=\zeta_{\widehat{S},K}^*(s)$ and, for each prime $p$, $\zeta_{S,K,p}^*(s)=\zeta_{\widehat{S},K,p}^*(s)$.
\end{proposition}
For a prime $p$, let $\mb{Z}_p$ be the ring of $p$-adic integers. Write $\widehat{N}=\prod_{p} \widehat{N}_p$, where $\widehat{N}_p$ is the pro-$p$ Sylow subgroup of $\widehat{N}$, and for a fixed prime $p$ let  $N_p'=\prod_{q\neq p} \widehat{N}_p$, where the product is taken over all primes $q$ different from $p$. Note that $N_p'$ is a closed characteristic subgroup of $\widehat{N}$,  $\widehat{N}_p=\widehat{N}/N_p'$ is the pro-$p$ completion of $N$ and, therefore, it is also the $\mb{Z}_p$-Mal'cev completion of $N$. Thus $S_p: 1\rightarrow \widehat{N}/N_p'\rightarrow \widehat{G}/N_p'\rightarrow F\rightarrow 1$ is the $\mb{Z}_p$-Mal'cev completion of $S$. We shall use the more comfortable notation $N_p=\widehat{N}_p/N_p'$, $G_p=\widehat{G}/N_p'$ and write
$S_p:1\rightarrow N_p\xrightarrow{\iota_p} G_p\xrightarrow{\pi_p} F\rightarrow 1$.

The following proposition was proved in \cite[Proposition 2.2]{dSMS} under stronger assumptions. The proof here is  essentially the same but we shall include it for the sake of completeness.
\begin{proposition}\label{expression as an Euler product of the relative local zeta functions}
Let $S:\SES$ be a virtually $\tau$-group. Then for each $K\in\mathscr{F}_F^*$ 
  \begin{align}
  \zeta_{S,K}^*(s)=\prod_{p\ \textrm{prime}}\zeta_{S,K,p}^*(s)=\prod_{p\ \textrm{prime}}\zeta_{S_p,K}^*(s).
  \end{align}
\end{proposition}
\begin{proof} By Proposition \ref{equality between zeta function of a group and that of its profinite completion}, to prove the first equality it is enough to show that $\zeta_{\widehat{S},K}^*(s)=\prod_p\zeta_{\widehat{S},K,p}^*(s)$. For $A\in\mathscr{F}^*_{\widehat{S},K}$, we have that $A\in \mathscr{F}_{\widehat{S},K,p}^*$ if and only if $A\cap \widehat{N}$ has index a power of $p$ in $\widehat{N}$, and this is the case if and only if $A\cap \widehat{N}$ contains $N_p'$; therefore, the map $A\mapsto A/N_p'$ gives an isomorphism of lattices between $\mathscr{F}_{\widehat{S},K,p}^*$ and $\mathscr{F}_{S_p,K}^*$, and the second equality holds. 

Since each $N_p'$ is a closed normal subgroup of $\widehat{G}$, it follows that for any $A\in\mathscr{F}^*_{\widehat{S},K}$ we have  $AN_p' \in \mathscr{F}_{\widehat{S},K,p}^*$; therefore, we obtain a map  $\Psi^*:\mathscr{F}_{\widehat{S},K}^*\rightarrow\prod_p\mathscr{F}_{\widehat{S},K,p}^*$ given by $A\rightarrow(AN_p')_p$. Note that given $A\in \mathscr{F}_{\widehat{S},K}^*$, for all but a finite number of primes $p$ one has $AN_p'=\widehat{\pi}^{-1}(K)$.

Given $A\in \mathscr{F}_{\widehat{S},K}^*$, we have $A=\cap_p AN_p'$ because $[\widehat{\pi}^{-1}(K):A]=[\widehat{N}:A\cap\widehat{N}]=\prod_p[\widehat{N}:(A\cap\widehat{N})N_p']=\prod_p[\widehat{N}:AN_p'\cap\widehat{N}]=\prod_p[\widehat{\pi}^{-1}(K):AN_p']=[\widehat{\pi}^{-1}(K):\cap_p AN_p']$.
    Reciprocally, given $A_p\in \mathscr{F}_{\widehat{S},K,p}^*$, where $A_p=\widehat{\pi}^{-1}(K)$ for all but a finite number of primes $p$, then for $A=\cap_p A_p$ we have $[A\widehat{N}:A]=[\widehat{N}:A\cap\widehat{N}]=[\widehat{N}:\cap_p(A_p\cap \widehat{N})]=\prod_p[\widehat{N}:A_p\cap \widehat{N}]=\prod_p[\widehat{\pi}^{-1}(K):A_p]=[\widehat{\pi}^{-1}:\cap_p A_p]$. Therefore $A\widehat{N}=\widehat{\pi}^{-1}(K)$ and since we also obtained $[\widehat{\pi}^{-1}(K):A]=\prod_p[\widehat{\pi}^{-1}(K):AN_p']$ the obvious inclusion $AN_p'\subseteq A_p$ implies $A_p=AN_p'$. 

   We conclude that the map $A\mapsto (AN_p')_p$ gives a bijection between $\mathscr{F}_{\widehat{S},K}^*$ and the set of those $(A_p)_p\in\prod_p\mathscr{F}_{\widehat{S},K,p}^*$ such that $A_p=\widehat{\pi}^{-1}(K)$ for all but a finite number of primes $p$, and moreover $[\widehat{\pi}^{-1}(K):A]=\prod_p[\widehat{\pi}^{-1}(K):AN_p']$. This is exactly the translation of the Euler product we wanted to prove.
\end{proof}

\subsection{Expression of $\zeta_{S,K}^*(s)$ as an Euler product of cone integrals} From now on we shall fix a virtually $\tau$-group $S:\SES$, where $\iota$ is an inclusion map, and let $h=h(N)$ be the Hirsch length of $N$. We shall keep the notation of the last subsection.

\begin{theorem}\label{main theorem} For each $K\in\mathscr{F}_F^*$, there are some cone integral data $\mc{D}^*$ such that, for each prime $p$,
\[\zeta_{S_p,K}^*(s)=(1-p^{-1})^{-h}Z_{\mc{D}^*}(s-h-|K|+1,p).\]
\end{theorem}

\noindent The method we shall use to express $\zeta_{S_p,K}^*(s)$ as a $p$-adic integral is essentially the same as the one used in \cite{dS} to study zeta functions of compact $p$-adic analytic groups (extensions of uniform pro-$p$ groups by finite groups).

Fix a Mal'cev basis $(x_1,\ldots,x_h)$ for $N$. For each $i=1,\ldots,h$ let us denote by $N^{(i)}$ and $N^{(i)}_p$ respectively the sets $\{x_i^{a_i}\ldots x_h^{a_h}:a_i,\ldots,a_h\in\mb{Z}\}$
and $\{x_i^{a_i}\ldots x_h^{a_h}:a_i,\ldots,a_n\in\mb{Z}_p\}$. By definition of Mal'cev basis, each $N^{(i)}$ (resp.\ $N_p^{(i)}$) is a normal subgroup of $N$ (resp.\ $N_p$), $N^{(i)}/N^{(i+1)}\cong \mb{Z}$ (resp.\ $N_p^{(i)}/N_p^{(i+1)}\cong\mb{Z}_p$); $N^{(i)}_p$ is the closure of $N^{(i)}$ in $N_p$ and it can be identified with the pro-$p$ completion of $N^{(i)}$. Moreover, $(x_i,\ldots,x_h)$ is a Mal'cev basis for $N_p^{(i)}$. 
For an $h$-tuple $\a=(a_1,\ldots,a_h)\in\mb{Z}_p^h$, we write $\x^\a=x_1^{a_1}\ldots x_h^{a_h}\in N_p$.

 For each open subgroup $B\leq N_p$, there exists an upper triangular matrix $\t\in Tr(h,\mb{Z}_p)$ such that $B=\{(\x^{\t_1})^{\lambda_1}\ldots(\x^{\t_h})^{\lambda_h}:\lambda_1,\ldots,\lambda_h\in\mathbb{Z}_p\}$ (here $\t_1,\ldots,\t_h$ are the row vectors of $\t$). We shall say that such a $\t$ represents a \emph{good basis} for $B$. It is easy to see that $B^{(i)}:=B\cap N_p^{(i)}=\{(\x^{\t_i})^{\lambda_i}\ldots (\x^{\t_h})^{\lambda_h}:\lambda_i,\ldots,\lambda_h\in\mb{Z}_p\}$.  Let $\mc{M}_p(B)$ denote the set of all those $\t\in Tr(h,\mb{Z}_p)$ representing a good basis for $B$. The following facts are proved in \cite[Section 2]{GSS}.
 \begin{enumerate}
  \item {$\mc{M}_p(B)$ is an open subset of $Tr(h,\mb{Z}_p)$.}
   \item {For $\t\in \mc{M}_p(B)$, the norm $|t_{ii}|_p$ depends only on $B$.}
    \item {$\mu(\mc{M}_p(B))=(1-p^{-1})^h\prod_{i=1}^h|t_{ii}|_p^i$, where $\mu$ is the Haar measure on $Tr(h,\mb{Z}_p)$ normalized in such a way that $\mu(Tr(h,\mb{Z}_p))=1$.}
    \item {$[N_p:B]=\prod_{i=1}^h|t_{ii}|^{-1}_p$.}
 \end{enumerate}

\noindent Denote by $\mu_{N_p}$ and $\mu_{\mathbb{Z}_p^h}$ the Haar measure on $N_p$ and $\mathbb{Z}_p^h$ respectively, normalized in such a way that $\mu_{N_p}(N_p)=\mu_{\mb{Z}_p^h}(\mb{Z}_p^h)=1$.

 \begin{lemma}\label{medida}  The map $\varphi:\a\mapsto\x^\a$ is a homeomorphism from $\mb{Z}_p^h$ onto $N_p$ which preserves measure, that is, such that $\mu_{N_p}(\varphi(S))=\mu_{\mb{Z}_p^h}(S)$ for every open subset $S$ of $\mb{Z}_p^h$.
 In particular, if  $B$ is an open subgroup of $N_p$ and $x\in N_p$ then $\mu(\{\a\in\mb{Z}_p^h:\x^\a\in xB\})=[N_p:B]^{-1}$.
 \end{lemma}
 \begin{proof} For $x\in N_p$ and an open subgroup $B$ of $N_p$ we clearly have that $\mu_{N_p}(xB)=[N_p:B]^{-1}$, and since these sets $xB$ form a basis for the topology of $N_p$ it is left to prove the last assertion. This will be done by induction on $h$.
 The case $h=1$ being trivial we suppose that $h>1$. Let $x=x_1^{\delta_1}\ldots x_h^{\delta_h}\in N_p$ and $B\leq N_p$ an open subgroup of $N_p$. Choose $\t\in Tr(h,\mb{Z}_p)$ representing a good basis for $B$. Then
 \begin{align*}
 xB&=\bigcup_{\lambda\in\mb{Z}_p}x_1^{\delta_1}\ldots x_h^{\delta_h}(x_1^{t_{11}}\ldots x_h^{t_{hh}})^\lambda B^{(2)}=\bigcup_{\lambda\in\mb{Z}_p}x_1^{\delta_1+\lambda t_{11}}x_2^{p(\lambda)}\ldots x_h^{p_h(\lambda)}B^{(2)},
\end{align*}  
where $p_2(\lambda),\ldots,p_h(\lambda)$ are polynomial functions in $\lambda$ with rational coefficients \cite[Theorem 6.5]{H}. We write $x(\lambda)=x_2^{p(\lambda)}\ldots x_h^{p_h(\lambda)}\in N_p^{(2)}$. Then
\begin{align*}
\varphi^{-1}(xB)&=\{(a_1,\ldots,a_h)\in\mb{Z}_p^h:\exists \lambda\in\mb{Z}_p,  a_1=\delta_1+\lambda t_{11},x_2^{a_2}\ldots x_h^{a_h}\in x(\lambda)B^{(2)}\}\\
&=\{(a_1,\ldots,a_h)\in\mb{Z}_p^h:a_1\in \delta_1+t_{11}\mb{Z}_p,\ x_2^{a_2}\ldots x_h^{a_h}\in x(\frac{a_1-\delta_1}{t_{11}})B^{(2)}\}
\end{align*}
The inductive hypothesis implies that $\mu_{\mb{Z}_p^{h-1}}(\{(a_2,\ldots,a_h)\in\mb{Z}_p^{h-1}:x_2^{a_2}\ldots x_h^{a_h}\in x(\frac{a_1-\delta_1}{t_{11}})B^{(2)}\})=[N_p^{(2)}:B^{(2)}]^{-1}$, and because $\mu_{\mb{Z}_p}(\delta_1+t_{11}\mb{Z}_p)=|t_{11}|_p=[N_p:BN_p^{(2)}]^{-1}$ we conclude that $\mu_{\mb{Z}^h_p}(xB)=[N_p:BN_p^{(2)}]^{-1}[N_p^{(2)}:B\cap N_p^{(2)}]^{-1}=[N_p:B]^{-1}$.
 \end{proof}
 
Now, choose a subset $\{g_f\in G: f\in F, \pi(g_f)=f,g_1=1,g_{f^{-1}}=g_f^{-1}\}\subseteq G$.
For $A\in\mathscr{F}_{S_p,K}$, it is easy to see that there exist elements $n_f\in N_p$, $f\in K$, such that \[ A=(A\cap N_p)\cup\bigcup_{f\in K,\ f\neq 1}g_fn_f(A\cap N_p).\]
This allows us to define $\mc{T}_p(A)$ as the set of all pairs of matrices  $(\t,\v)\in Tr(h,\mb{Z}_p)\times M_{K\setminus\{1\}\times h}(\mb{Z}_p)$ (here $M_{K\setminus\{1\}\times h}(\mb{Z}_p)$ stands for $\prod_{f\in K\setminus\{1\}}\mb{Z}_p^h$) such that $\t$ represents a good basis for $ A\cap N_p$ and $\ds\{1\}\cup\bigcup_{f\in K\setminus\{1\}}\{g_f\x^{\v_f}\}$ is a transversal for the cosets of $A\cap N_p$ in $A$ (here $\v_f\in\mathbb{Z}_p^h$ is the $f$-th row vector of the matrix $\v\in M_{K\setminus\{1\}\times h}(\mb{Z}_p)$).
\begin{lemma}\label{medida de T(A)} $\mc{T}_p(A)$ is an open subset of $Tr(h,\mb{Z}_p)\times M_{K\setminus\{1\}\times h}(\mb{Z}_p)$ with Haar measure
\[\mu(\mc{T}_p(A))=(1-p^{-1})^h\prod_{i=1}^h|t_{ii}|_p^{i+|K|-1}.\]\end{lemma}
\begin{proof} If $A=(A\cap N_p)\cup\bigcup_{f\in K\setminus\{1\}} g_fn_f(A\cap N_p)$ with $n_f\in N_p$, then clearly $(\t,\v)\in \mc{T}_p(A)$ if and only if $\t\in\mc{M}_p(A\cap N_p)$ and $g_f\x^{\v_f}(A\cap N_p)=g_fn_f(A\cap N_p)$ for all $f\in K\setminus\{1\}$. The last equality is equivalent to $\x^{\v_f}\in n_f(A\cap N_p)$ and this is again equivalent to $\v_f\in \varphi^{-1}(n_f(A\cap N_p))$. Therefore $\mc{T}_p(A)=\mc{M}_p(A\cap N_p)\times \prod_{f\in K\setminus\{1\}}\varphi^{-1}(n_f(A\cap N_p))$ and this is an open subset of $Tr(h,\mb{Z}_p)\times M_{K\setminus\{1\}\times h}(\mb{Z}_p)$. Using Lemma \ref{medida} and the formulae for $\mu(\mc{M}_p(A\cap N_p))$ and $[N_p:A\cap N_p]$ listed above, one obtains the desired formula for $\mu(\mc{T}_p(A))$.
\end{proof}

Defining $\mc{T}_{K,p}^*=\bigcup_{A\in\mathscr{F}_{S_p,K}^*}\mc{T}_p(A)$ and arguing as in the proof of Proposition 2.6 of \cite{GSS}, one obtains
\begin{equation}\label{formula integral}
\zeta_{S_p,K}^*(s)=(1-p^{-1})^{-h}\int_{\mathcal{T}_{K,p}^*}\prod_{i=1}^h|t_{ii}|_p^{s-i-|K|+1}d\mu.
\end{equation}

  The next step is to describe  $\mc{T}_{K,p}^\s$ and $\mc{T}_{K,p}^\lhd$ as sets of matrices with entries satisfying cone conditions, that is, we want to find a finite set of polynomials $p_i^*,q_i^*$ in an appropriate number of variables with rational coefficients such that, for each prime $p$, $\mc{T}_{K,p}^*=\{(\t,\v)\in Tr(h,\mb{Z}_p)\times M_{K\setminus\{1\}\times h}(\mb{Z}_p):p_i^*(\t,\v)|q_i^*(\t,\v)\}$ up to a set of measure zero.
 One condition for a pair $(\t,\v)$ to be in $\mc{T}_{K,p}^*$ is that $\t$ must represent a good basis for some $B\in\mathscr{F}^*_{N_p}$. We will see that this condition can be described using cone conditions, and then we will see that the other conditions on $(\t,\v)$ to be in $\mc{T}_{K,p}^*$ are a finite number of conditions of the form $\x^{\h(\t,\v)}\in\overline{\langle \x^{\t_1},\ldots,\x^{\t_h}\rangle}\subseteq N_p$, where $\h=(h_1,\ldots,h_h)$ and each $h_i$ is a polynomial with rational coefficients not depending on $p$ (here the bar $\overline{\ \cdot\  }$ denotes topological closure in $N_p$). Therefore we have to translate this kind of conditions into cone conditions and this is what we will do first.

 At this point, we need to set up some further notations and conventions. All the polynomials will have coefficients in $\mb{Q}$. There will be polynomials in many different numbers of variables and instead of trying to define them precisely, we shall deduce their definitions from where they are being evaluated. For example, if $q$ is a polynomial which is evaluated on the entries of upper triangular matrices $\t\in Tr(h,\mb{Z}_p)$, then it means that $q$ is a polynomial in the variables $T_{ij}$ ($1\leq i\leq j\leq h$) with rational coefficients. We shall write simply $q(\T)\in \mb{Q}[\T]$ and its evaluation on the entries of $\t$ will be written $q(\t)$. Similarly, if a polynomial $q$ is evaluated on the entries of pair of matrices $(\t,\v)\in Tr(h,\mb{Z}_p)\times M_{K\setminus\{1\}}(\mb{Z}_p)$, then we shall write $q(\T,\V)\in\mb{Q}[\T,\V]$. As a last example, if $\t'$ is the matrix which is obtained from $\t$ by deleting the first row and the first column, then a polynomial which can be evaluated on the entries of $\t'$ will be written $p(\T')\in\mb{Q}[\T']$.

For a vector $\v=(v_1,\ldots,v_h)$, we shall write $\v^i=(0,\ldots,0,v_i,\ldots,v_h)$.
For an $h_1\times h_2$-matrix $\t$, the vectors $\t_1,\ldots,\t_{h_1}$ denote the row vectors of $\t$ and, for $j\leq\min \{h_1,h_2\}$, $\t^{(j)}$ denotes the matrix obtained from $\t$ by deleting its first $j-1$ rows and its  first $j-1$ columns. An $h$-tuple of polynomials $(p_1(\T),\ldots,p_h(\T))$ will be written as $\mathbf{p}(\T)$, and so the notation $\x^{\mathbf{p}(\t)}$  means just $x_1^{p_1(\t)}\ldots x_h^{p_h(\t)}$.

Finally, we recall the polynomials expressing multiplication and exponentiation in the $\tau$-group $N$ with respect to the given Mal'cev basis $(x_1,\ldots,x_h)$ (see \cite[Chapter 6]{H}). These are $h$-tuples of polynomials $f_1(\X,\Y),\ldots,f_h(\X,\Y)\in \mb{Q}[\X,\Y]$ and $g_1(\X,W),\ldots,g_h(\X,W)\in\mb{Q}[\X,W]$ such that $\mathbf{x}^{\mathbf{a}}\mathbf{x}^{\mathbf{b}}=\mathbf{x}^{\mathbf{f}(\mathbf{a},\mathbf{b})}$ and $(\mathbf{x}^{\mathbf{a}})^w=\mathbf{x}^{\mathbf{g}(\mathbf{a},w)}$ for all $\mathbf{a},\mathbf{b}\in\mb{Z}^h$ and $w\in\mb{Z}$. Using these polynomials we can also express the commutator operation by polynomials, that is, there exist polynomials $c_1(\X,\Y),\ldots,c_h(\X,\Y)\in\mb{Q}[\X,\Y]$ such that $[\x^{\a},\x^{\b}]=\x^{\mathbf{c}(\a,\b)}$.
Some practical facts we shall use about these polynomials are: $f_1(\X,\Y)=X_1+Y_1$, and more generally $f_i(\a^i,\b^i)=a_i+b_i$, $\forall i$. Similarly $g_i(\a^i,w)=a_iw$ and $c_k(\a^i,\b^j)=0$, $\forall k\leq\max\{i,j\}, \forall i,j$.

\emph{Algorithm:} We define recursively polynomials $p_i,q_i\in\mb{Q}[\T,\Z]$ for $i=1,\ldots,h$ such that, for each prime $p$, if $\t$ represents a good basis for some open subgroup of $N_p$, then the condition $\x^{\z}\in\overline{\langle\x^{\t_1},\ldots,\x^{\t_h}\rangle}$ is equivalent to $q_i(\t,\z)|p_i(\t,\z),\ i=1,\ldots,h$:

\ \noindent 1.\ Choose variables $T_{ij}$ for $1\leq i\leq j\leq h$, $Z_1,\dots,Z_h$ and $W_1,\ldots,W_h$, and set $\T_i=(0,\ldots,0,T_{ii},\ldots,T_{ih})$ and $\Z=(Z_1,\ldots,Z_h)$.

\ \noindent 2.\ Define recursively $h$-tuples of polynomials $\k_1,\ldots\k_h$ by
\begin{itemize}
\item $\k_1=\Z$,
\item $\k_i=\f(\mathbf{g}(\mathbf{g}^{i}(\T_{i-1},W_{i-1}),-1),\k_{i-1}^{i})$ for $1<i\leq h$.
\end{itemize}
Observe that the entries of $\k_i$ are polynomials in the variables $T_{rs}$ ($r<i,\ r\leq s$), $Z_1,\ldots,Z_h, W_1,\ldots,W_{i-1}$. For simplicity we shall write $\k_i=\k_i(\T_1,\ldots,\T_{i-1},\Z,W_1,\ldots,W_{i-1})$.

\ \noindent 3.\ Define rational functions $v_1(\T,\Z),\ldots,v_h(\T,\Z)$ recursively by
\begin{itemize}
\item $v_1(\T,\Z)=k_{11}(\Z)/T_{11}$
\item $v_i(\T,\Z)=k_{ii}(\T_1,\ldots,\T_{i-1},\Z,v_1(\T,\Z),\ldots,v_{i-1}(\T,\Z))/T_{ii}$ for $1<i\leq h$.
\end{itemize}
Write $v_i(\T,\Z)=p_i(\T,\Z)/q_i(\T,\Z)$ with $p_i(\T,\Z),q_i(\T,\Z)\in\mb{Q}[\T,\Z]$.
Note that $ q_i(\T,\Z)$ can be chosen to be a monomial in $T_{11},\ldots,T_{ii}$.
\begin{proposition}\label{Algoritmo}
   Suppose that $\t\in Tr(h,\mb{Z}_p)$  represents a good basis for an open subgroup of $N_p$ and let $\z=(z_1,\ldots,z_h)\in\mb{Z}_p^h$.
   Then $x_1^{z_1}\ldots x_h^{z_h}\in\overline{\langle \x^{\t_1},\ldots,\x^{\t_h}\rangle}$ if and only if $q_i(\t,\z)|p_i(\t,\z)$ for $i=1,\ldots,h$.
\end{proposition}

\begin{proof}
Let $B_\t=\overline{\langle\x^{\t_1},\ldots,\x^{\t_h}\rangle}$. Since $\t$ represents a good basis for $B_\t$, every element of $B_\t$ can be written uniquely in the form $(\x^{\t_1})^{a_1}\ldots(\x^{\t_h})^{a_h}$ for some $a_i\in\mb{Z}_p$. The element $(\x^{\t_i})^{a_i}\ldots (\x^{\t_h})^{a_h}$, written in the Mal'cev basis $(x_1,\ldots,x_h)$, has the form $x_i^{t_{ii}a_i}x_{i+1}^{b_{i+1}}\ldots x_h^{b_h}$ for some $b_{i+1},\ldots,b_h\in\mb{Z}_p$. We shall make use of these facts without mention.

Let $w_i=v_i(\t,\z)$. We have that $\x^{\z}\in B_\t$ if and only if  $\x^{\z}=(\x^{\t_1})^{a_1}\ldots (\x^{\t_h})^{a_h}$ for some $a_1,\ldots,a_h\in\mb{Z}_p$. Since $(\x^{\t_1})^{a_1}\ldots (\x^{\t_h})^{a_h}$ has the form $x_{1}^{t_{11}a_1}x_2^{b_2}\ldots x_h^{b_h}$, the condition $\x^\z\in B_\t$ implies $t_{11}|z_1$, or equivalently $w_1\in\mb{Z}_p$, and therefore the element $(\x^{\t_1})^{z_1/t_{11}}=\x^{\mathbf{g}(\t_1,w_1)}$ must lie in $B_\t$. Then $\x^{\z}\in B_\t$ if and only if $w_1\in\mb{Z}_p$ and $\x^{\z}=\x^{\mathbf{g}(\t_1,w_1)}a$ for some $a\in B_\t$. Since $g_1(\t_1,w_1)=t_{11}w_1=z_1$, the last condition is equivalent to $w_1\in\mb{Z}_p$ and $\x^{\z^2}=\x^{\mathbf{g}^2(\t_1,w_1)}a$, for some $a\in\overline{\langle \x^{\t_2},\ldots,\x^{\t_h}\rangle}$.
Operating with $\f$ and $\mathbf{g}$ we get
     \[(\x^{\mathbf{g}^2(\t_1,w_1)})^{-1}\x^{\z^2}=\x^{\f(\mathbf{g}(\mathbf{g}^{2}(\t_1,w_1),-1),\z^{2})}=\x^{\k_2(\t_1,\z,w_1)}.\]
Thus $\x^{\z}\in B_\t$ if and only if $w_1\in\mb{Z}_p$ and $\x^{\k_2(\t_1,\z,w_1)}\in\overline{\langle\x^{\t_2},\ldots,\x^{\t_h}\rangle}$.

   Let $i>1$ and assume that $\x^{\z}\in B_\t$ if and only if $w_1,\ldots,w_{i-1}\in\mb{Z}_p$ and $\x^{\k_i(\t_1,\ldots,\t_{i-1},\z,w_1,\ldots,w_{i-1})}\in\overline{\langle \x^{\t_i},\ldots,\x^{\t_h}\rangle}$.
   Working as in the last paragraph and assuming that $w_1,\ldots,w_{i-1}\in\mb{Z}_p$, we see that the last condition  is equivalent to $t_{ii}|k_{ii}(\t_1,\ldots,\t_{i-1},\z,w_1,\ldots,w_{i-1})$ (or equivalently $w_i\in\mb{Z}_p$) and $$\x^{\k_i(\t_1,\ldots,\t_{i-1},\z,w_1,\ldots,w_{i-1})}=\x^{\mathbf{g}(\t_{i},w_i)}a\ \ \mbox{for some }\ a\in\overline{\langle \x^{\t_i},\ldots,\x^{\t_h}\rangle}.$$ 
   Again this is equivalent to $w_i\in\mb{Z}_p$ and $$\x^{\k_i^{i+1}(\t_1,\ldots,\t_{i-1},\z,w_1,\ldots,w_{i-1})}=\x^{\mathbf{g}^{i+1}(\t_i,w_i)}a\ \ \mbox{ for some }\ a\in \overline{\langle \x^{\t_{i+1}},\ldots,\x^{\t_h}\rangle}.$$
 If we write $(\x^{\mathbf{g}^{i+1}(\t_i,w_i)})^{-1}\x^{\k_i^{i+1}(\t_1,\ldots,\t_{i-1},\z,w_1,\ldots,w_i)}$ as $\x^\a$ then clearly
  $$\a=\f(\mathbf{g}(\mathbf{g}^{i+1}(\t_i,w_i),-1),\k_i^{i+1}(\t_1,\ldots,\t_{i-1},\z,w_1,\ldots,w_i)),$$ which is $\k_{i+1}(\t_1,\ldots,\t_i,\z,w_1,\ldots,w_i)$. It follows that $\x^\z\in B_\t$ if and only if $w_1,\ldots,w_i\in\mb{Z}_p$ and $\x^{\k_{i+1}(\t_1,\ldots,\t_i,\z,w_1,\ldots,w_i)}\in\overline{\langle \x^{\t_{i+1}},\ldots,\x^{\t_h}\rangle}$. By induction we conclude that $\x^\z\in B_\t$ if and only if $w_1,\ldots,w_h\in\mb{Z}_p$, and since $w_i=p_i(\t,\z)/q_i(\t,\z)$, we have that $\x^\z\in B_\t$ if and only if $q_i(\t,\z)|p_i(\t,\z)$ for $i=1,\ldots,h$.
\end{proof}

\begin{corollary}\label{Algoritmo, corollary}
Let $N$ be a $\tau$-group with Mal'cev basis $(x_1,\ldots,x_h)$, and for each prime $p$ let $\mc{M}_p^{\s}=\cup_{B\in\mathscr{F}_{N_p}^\s}\mc{M}_p(B)$ (calculated with respect to this basis). Then there exist a finite set $I$ and polynomials $r_i(\T),s_i(\T)\in\mb{Q}[\T]$ ($i\in I$) such that, for each prime $p$, $\mc{M}_p^\s=\{\t\in Tr(h,\mb{Z}_p): t_{11}\ldots t_{hh}\neq 0,\ s_i(\t)|r_i(\t),\forall i\in I\}$.
\end{corollary}
\begin{proof} We proceed by induction on $h$. The case  $h=1$ being trivial we suppose that $h>1$. Recall the definitions of $N^{(j)}$ and $N_p^{(j)}$  ($j=1,\ldots,h$) at the beginning of this subsection, and observe that if
 $\t$ represents a good basis for $N^{(j)}_p$, then $\t^{(j')}$ represents a good basis for $N^{(j+j'-1)}_p$.
Applying the algorithm to $N^{(j)}$ with Mal'cev basis $(x_j,\ldots,x_h)$, we obtain polynomials $p_i^{j}(\T^{(j)},\Z^j), q_i^{j}(\T^{(j)},\Z^j)\in\mb{Q}[\T^{(j)},\Z^j]$, $i=j,\ldots,h$.

The inductive hypothesis, applied to $N^{(2)}$ with Mal'cev basis $(x_2,\ldots,x_h)$, implies that there exist a finite set $I_1$ and polynomials $s_i(\T^{(2)}),t_i(\T^{(2)})$ ($i\in I_1$) such that $\t'\in Tr(h-1,\mb{Z}_p)$ represents a good basis for some open subgroup of $(N^{(2)})_p$ if and only if $t'_{22}\ldots t'_{hh}\neq 0$ and $s_i(\t')|r_i(\t')$ for all $i\in I_1$. It is proved in \cite[Lemma 2.10]{L} that
a matrix $\t\in Tr(h,\mb{Z}_p)$ represents a good basis for some open subgroup of $N_p$ if and only if
\begin{enumerate}
  \item $t_{11}\ldots t_{hh}\neq 0$ and
  \item $\x^{\mathbf{c}(\t_i,\t_j)}\in \overline{\langle\x^{\t_{j+1}},\ldots,\x^{\t_h}\rangle}$ for all $1\leq i<j\leq h$;
\end{enumerate}
this is equivalent to
\begin{enumerate}
  \item $t_{11}\ldots t_{hh}\neq 0$,
  \item $\t^{(2)}$ represents a good basis for $N^{(2)}_p$, and
  \item $\x^{\mathbf{c}(\t_1,\t_j)}\in\overline{\langle\x^{\t_{j+1}},\ldots,\x^{\t_{h}}\rangle}$ for $j=2,\ldots,h$;
\end{enumerate}
and this  is again equivalent to
\begin{enumerate}
  \item $t_{11}\ldots t_{hh}\neq 0$,
  \item $s_i(\t)|r_i(\t), i\in I_1$, and
  \item $q_i^j(\t^{(j+1)},\mathbf{c}(\t_1,\t_j))|p_i^j(\t^{(j+1)},\mathbf{c}(\t_1,\t_j)), i=j,\ldots,h.$, $j=2,\ldots,h$.
\end{enumerate}
In the last equivalence we are using Proposition \ref{Algoritmo} in combination with the inductive hypothesis.
This completes the proof.
\end{proof}

\begin{corollary}\label{Algoritmo,corollary 2}
The subgroup zeta function and the normal zeta function of a $\tau$-group are Euler products of  cone integrals over $\mb{Q}$.
\end{corollary}
\begin{proof}
We fix a $\tau$-group $N$ with a Mal'cev basis as before. It follows from \cite[Proposition 2.6]{GSS} that
\[\zeta_N^*(s)=\prod_p \left((1-p^{-1})^{-h}\int_{\mc{M}_p^*}\prod_{i=1}^h|t_{ii}|_p^{s-i}d\mu\right),\]
where $\mc{M}_p^*$ is the set of all those $\t\in Tr(h,\mb{Z}_p)$ representing a good basis for some $B\in\mathscr{F}_{N_p}^*$.
 By Corollary \ref{Algoritmo, corollary}, there exist a finite set $I$ and polynomials $r_i(\T),s_i(\T)\in\mb{Q}[\T]$ ($i\in I$) such that, for each prime $p$, $\mc{M}_p^\s=\{\t\in Tr(h,\mb{Z}_p): t_{11}\ldots t_{hh}\neq 0, s_i(\t)|r_i(\t) \forall i\in I\}$.  Since the set of matrices $\t\in Tr(h,\mb{Z}_p)$ with $t_{11}\ldots t_{hh}=0$ has Haar measure zero,  we can eliminate the condition $t_{11}\ldots t_{hh}\neq 0$ from the set of integration. This proves the corollary in the case $*=\leq$. In the case $*=\lhd$, it follows from  \cite[Lemma 2.4]{GSS} that $$\mc{M}_p^\lhd=\{\t\in\mc{M}_p^\s: \x^{\mathbf{c}(\mathbf{e}_i,\t_j)}\in\overline{\langle\x^{\t_{1}},\ldots\x^{\t_h}\rangle}, \forall i,j\}$$ (here $\mathbf{e}_1,\ldots,\mathbf{e}_h$ are the unit basis vectors of $\mb{Z}^h$). Using the polynomials $p_i(\T,\Z)$ and $q_i(\T,\Z)$ from the algorithm, Proposition \ref{Algoritmo} and the last description of $\mc{M}_p^\s$, we obtain
\begin{align*} \mc{M}_p^\lhd=&\{\t\in Tr(h,\mb{Z}_p): t_{11}\ldots t_{hh}\neq 0,\ s_i(\t)|r_i(\t)\ \forall i\in I,\\
 & q_i(\t,\mathbf{c}(\mathbf{e}_i,\t_j))|p_i(\t,\mathbf{c}(\mathbf{e}_i,\t_j))\ \forall i,j=1,\ldots,h\}.
\end{align*}
 Again, we can eliminate the condition $t_{11}\ldots t_{hh}\neq 0$ from the set of integration. This completes the proof.
  \end{proof}

We come back to our original situation. We have fixed a virtually $\tau$-group $S:\SES$ and a Mal'cev basis $(x_1,\ldots,x_h)$ for $N$. We have chosen elements $g_f\in G$ for $f\in F$ such that $\pi(g_f)=f$, $g_1=1$ and $g_{f^{-1}}=g_f^{-1}$. Let $K\in\mathscr{F}_{F}^*$. For each prime $p$, we defined $S_p$ and we obtained an expression for $\zeta_{S_p,K}(s)$ as a $p$-adic integral over an open subset $\mc{T}_{K,p}^*\subseteq Tr(h,\mb{Z}_p)\times M_{f\in K\setminus\{1\}\times h}(\mb{Z}_p)$.
\begin{proposition}\label{main theorem, application}
  Given $*\in\{\leq,\lhd\}$, there exist a finite set $S$ and $h$-tuples of polynomials $\z_s^*(\T,\V)=(z_{s1}^*(\T,\V),\ldots,z_{sh}^*(\T,\V))$ ($s\in S$) with $z_{si}^*(\T,\V)\in\mb{Q}[\T,\V]$, such that for each prime $p$
\begin{align*}\mc{T}_{K,p}^*=\{(\t,\v)\in \mbox{Tr}(h,\mb{Z}_p)\times M_{K\setminus\{1\}\times h}(\mb{Z}_p):&\ \t\in\mc{M}_p^*,\ \mbox{and}\  \forall\ s\in S \\ &\x^{\mathbf{z}_s^*(\t,\v)}\in\overline{\langle \x^{\t_1},\ldots,\x^{\t_h}\rangle}\}.\end{align*}
\end{proposition}
\begin{proof}
     By definition, $(\t,\v)\in\mathcal{T}_{K,p}^\s$ if and only if $\t$ represents a good basis for an open subgroup $B_{\t}$ of $N_p$ and $\ds A_{(\t,\v)}:=B_{\t}\cup\bigcup_{f\in K\setminus\{1\}} g_f\x^{\v_f} B_{\t}$ lies in $\mathscr{F}_{S_p,K}^\s$. Then we can assume that $\t$  represents a good basis for an open subgroup $B_\t$.
 We claim that $A_{(\t,\v)}$ is a subgroup of $G_p$ if and only if the following hold.
\begin{enumerate}
\item{$(g_f\x^{\v_f})^{-1}\x^{\t_i}g_f\x^{\v_f}\in B_{\t}$ for $i=1,\ldots,h$ and $f\in K\setminus\{1\}$;}
\item{for $f,f'\in K\setminus\{1\}$ with $ff'\neq 1$, $(g_{ff'}\x^{\v_{ff'}})^{-1}g_f\x^{\v_f}g_{f'}\x^{\v_{f'}}\in B_{\t}$;}
\item{$g_f\x^{\v_f}g_{f^{-1}}\x^{\v_{f^{-1}}}\in B_{\t}$ for $f\in K\setminus\{1\}$.}
\end{enumerate}
All these conditions are necessary because if $A_{(\t,\v)}$ is a subgroup of $G_p$ then $B_\t=A_{(\t,\v)}\cap N_p$ is a normal subgroup of $A_{(\t,\v)}$. Thus (1) is consequence of this; (2) and (3) reflect the fact that $A_{(\t,\v)}/B_{\t}$ is a group. Conversely, (1) implies that $A_{(\t,\v)}$ is contained in the normalizer of $B_\t$; (2) and (3) imply that $ A_{(\t,\v)}/B_\t$ is a group and therefore $A_{(\t,\v)}$ is a subgroup of $G_p$. If this is the case, then $A_{(\t,\v)}\in\mathscr{F}_{S_p,K}^\s$.

If $A_{(\t,\v)}$ is a subgroup of $G_p$, we claim that it is a normal subgroup in $G_p$ if and only if the following holds
\begin{enumerate}
\setcounter{enumi}{3}
\item $\x^{\mathbf{c}(\mathbf{e}_i,\t_j)}\in B_\t$, for $i,j=1,\ldots,h$, where $\mathbf{e}_1,\ldots,\mathbf{e}_h$ are the canonical vectors of $\mb{Z}^h$;
\item {$g_f^{-1}\x^{\t_i}g_f\in B_{\t}$ for $i=1,\ldots,h$ and $f\in F\setminus\{1\}$;}
\item {$(\x^{\v_f})^{-1} g_f^{-1}x_ig_f\x^{\v_f}x_i^{-1}\in B_\t$, for $i=1,\ldots,h$ and $f\in K\setminus\{1\}$;}
\item {$(\x^{\v_{ff'f^{-1}}})^{-1}g_{ff'f^{-1}}^{-1}(g_fg_{f'}g_f^{-1})(g_f\x^{\v_{f'}}g_f^{-1})\in B_\t,\ f\in F\setminus\{1\},f'\in K\setminus\{1\}$.}
\end{enumerate}
If $A_{(\t,\v)}$ is normal in $G_p$, then $B_\t=N_p\cap A_{(\t,\v)}$ is also normal in $G_p$. This implies (5) and with  \cite[Lemma 2.4]{GSS} we also obtain (4). Again, normality of $B_\t$ implies that
 \begin{align}\label{conjugation by x_i}
 x_i A_{(\t,\v)}x_i^{-1}=B_\t\cup\bigcup_{f\in K\setminus\{1\}} x_ig_f\x^{\v_f}x_i^{-1}B_\t
 \end{align}
 for all $i=1,\ldots,h$; and
 \begin{align}\label{conjugation by g_f}
 g_fA_{(\t,\v)}g_f^{-1}=B_\t\cup \bigcup_{f'\in K\setminus\{1\}}g_fg_{f'}\x^{\v_{f'}}g_f^{-1} B_\t
 \end{align}
 for all $f\in F\setminus\{1\}$.
 Doing some computations one obtains
\begin{align}\label{conjugation by x_i: formulas}
x_ig_f\x^{\v_f}x_i^{-1}=g_f\x^{\v_f}((\x^{\v_f})^{-1} g_f^{-1}x_ig_f\x^{\v_f}x_i^{-1})
\end{align}
and
 \begin{align}\label{conjugation by g_f: formulas}
g_fg_{f'}\x^{\v_{f'}}g_f^{-1}=g_{ff'f^{-1}}\x^{\v_{ff'f^{-1}}}((\x^{\v_{ff'f^{-1}}})^{-1}g_{ff'f^{-1}}^{-1}(g_fg_{f'}g_f^{-1})(g_f\x^{\v_{f'}}g_f^{-1})). \end{align}
Therefore, if $A_{(\t,\v)}$ is normal then (\ref{conjugation by x_i}) and (\ref{conjugation by x_i: formulas}) implies (6) and it also implies that $K=\pi_p(A_{(\t,\v)})$ is normal. Thus $\pi_p^{-1}(K)/N$ is normal and therefore $\{g_{ff'f^{-1}}\x^{\v_{ff'f^{-1}}}N\}_{f'\in K}=\{g_{f'}\x^{\v_{f'}}N\}_{f'\in K}$ for each $f\in F$. This, in combination with (\ref{conjugation by g_f}) and (\ref{conjugation by g_f: formulas}), implies (7). Conversely, (4) and (5) imply that $B_\t$ is normal in $G_p$ and then (\ref{conjugation by x_i}), (\ref{conjugation by x_i: formulas}), (\ref{conjugation by g_f: formulas}), (6) and (7) imply that $A_{(\t,\v)}$ is normal in $G_p$.

Now we shall put each one of the conditions (1)-(7) in the form $\x^{\mathbf{z}(\t,\v)}\in B_\t$ for some polynomial $\z(\T,\V)\in\mb{Q}[\T,\V]$.
 For $i=1,\ldots, h$, $f,f'\in F$, we put $g_f^{-1}x_ig_f=\x^{\mathbf{l}_{if}}$ and $g_fg_{f'}=g_{ff'}\x^{\mathbf{n}_{ff'}}$. For any $\u\in\mb{Z}^h$, we have  $g_f^{-1}\x^{\u}g_f=(\x^{\mathbf{l}_{1f}})^{u_1}\ldots (\x^{\mathbf{l}_{hf}})^{u_h}=\x^{\mathbf{g}(\mathbf{l}_{1f},u_1)}\ldots \x^{\mathbf{g}(\mathbf{l}_{hf},u_h)}=\x^{\f(\f(\ldots\f({\mathbf{g} (\mathbf{l}_{1f},u_1)},{\mathbf{g}(\mathbf{l}_{2f},u_2)}),\ldots),{\mathbf{g}(\mathbf{l}_{hf},u_h)})}=\x^{\mathbf{p}_{f}(\u)}$, for some $h$-tuple of  polynomials $\mathbf{p}_f(\U)$. Now we can check easily the following equalities:
 \begin{align*}
  (g_f\x^{\v_f})^{-1}\x^{\t_i}g_f\x^{\v_f}&=\x^{\f(\f(\mathbf{g}(\v_f,-1),\mathbf{p}_{f}(\t_i)),\v_f)}\\
   (g_{ff'}\x^{\v_{ff'}})^{-1}g_f\x^{\v_i}g_{f'}\x^{\v_{f'}}&=\x^{\f(\f(\f(\mathbf{g}(\v_{ff'},-1),\mathbf{n}_{ff'}),\mathbf{p}_{f'}(\v_f)),\v_{f'})}\\
   g_f\x^{\v_f}g_{f^{-1}}\x^{\v_{f^{-1}}}&=\x^{\f(\mathbf{p}_{f^{-1}}(\v_f),\v_{f^{-1}})}\\
   \x^{\mathbf{c}(\mathbf{e}_i,\t_j)}&=\x^{\mathbf{c}(\mathbf{e}_i,\t_j)}\\
   g_f^{-1}\x^{\t_i}g_f&=\x^{\mathbf{p}_f(\t_i)} \\
   (\x^{\v_f})^{-1} g_f^{-1}x_ig_f\x^{\v_f}x_i^{-1}&=\x^{\f(\mathbf{g}(\v_f,-1),\f(\mathbf{l}_{if},\f(\v_f,-\mathbf{e}_i)))} \\
   (\x^{\v_{ff'f^{-1}}})^{-1}g_{ff'f^{-1}}^{-1}(g_fg_{f'}g_f^{-1})(g_f\x^{\v_{f'}}g_f^{-1})&=\x^{\f(\mathbf{g}(\v_{ff'f^{-1}},-1),\f(\mathbf{n}_{f(f'f^{-1})},\f(\mathbf{n}_{f'f^{-1}},\mathbf{p}_{f^{-1}}(\v_{f'}))))}
 \end{align*}
 Using these equalities we can transform each condition (1)-(7) into a condition of the form $\x^{\mathbf{z}(\t,\v)}\in B_\t$ as desired. This completes the proof.
   \end{proof}

\begin{proof}[Proof of Theorem \ref{main theorem}.]
Expression (\ref{formula integral}) gives the expression of $\zeta_{S_p,K}^*(s)$ as a $p$-adic integral over a subset $\mathcal{T}_{K,p}^*$ of $Tr(h,\mb{Z}_p)\times M_{K\setminus\{1\}\times h}(\mb{Z}_p)$; Corollary \ref{Algoritmo, corollary} and Proposition \ref{main theorem, application} give the description of $\mc{T}_{K,p}^*$ (up to a set of measure zero) with cone conditions where the polynomials are independent of $p$.\end{proof}

The proof of Theorem 1 is now complete. It follows from expression (\ref{expression as an Euler product of the relative local zeta functions}) and Theorem \ref{main theorem}.

It is proved in \cite[Corollary 3.2, Corollary 4.14 and Lemma 4.15]{dSG} that if a non-constant series $Z(s)$ is the Euler product of cone integrals over $\mb{Q}$ with cone integral data $\mc{D}$, that is, $Z(s)=Z_{\mc{D}}(s)$, then the abscissa of convergence of each $Z_{\mc{D}}(s,p)$ is strictly less than the abscissa of convergence of $Z_\mc{D}(s)$. Thus, if $Q$ is a finite set of primes, then $Z_\mc{D}(s)$ and $\prod_{p\notin Q} Z_{\mathcal{D}}(s,p)$ have the same abscissa of convergence (note that this is also the case if $Z(s)$ is constant). Using this and Theorem 1 we conclude
\begin{corollary}\label{Condition for having the same abscissa of convergence}
For $i=1,2$, let $S_i:1\rightarrow N_i\xrightarrow{\iota_i} G_i\xrightarrow{\pi_i} F_i\rightarrow 1$ be a virtually $\tau$-group, $K_i\in\mathscr{F}_{F_i}^*$ and let $Q$ be a finite set of primes. Then
\begin{enumerate}
  \item $\zeta_{S_1,K_1}^*(s)$ and $\prod_{p\notin Q}\zeta_{S_1,K_1,p}^*(s)$ have the same abscissa of convergence;
  \item if $\zeta_{S_1,K_1,p}^*(s)=\zeta_{S_2,K_2,p}^*(s)$ for all $p\notin Q$, then $\zeta_{S_1,K_1}^*(s)$ and $\zeta_{S_2,K_2}^*(s)$ have the same abscissa of convergence.
\end{enumerate}
\end{corollary}

\subsection{The abscissa of convergence as an invariant of the $\mb{Q}$-Mal'cev completion}
If $S^{\mb{Q}}: 1\rightarrow N^\mb{Q}\rightarrow G^\mb{Q}\rightarrow F\rightarrow 1$ is the $\mb{Q}$-Mal'cev completion of some virtually $\tau$-group $S$, then the isomorphism class of $S^{\mb{Q}}$ is completely determined by $G^\mb{Q}$. This is because $N^\mb{Q}$ is the unique maximal $\mb{Q}$-radicable subgroup of $G^\mb{Q}$. 

Now we fix the group $G^\mb{Q}$ from $S^\mb{Q}$ and we consider $N^\mb{Q}$ as a subgroup of $G^\mb{Q}$. Given a subgroup $G_1$ of $G^\mb{Q}$ such that $N^{\mb{Q}}G_1=G^\mb{Q}$ and such that $N_1:=N^{\mb{Q}}\cap G_1$ is a $\tau$-group with $\mb{Q}$-Mal'cev completion $N^{\mb{Q}}$, then the short exact sequence $S_1:1\rightarrow N_1\rightarrow G_1\rightarrow F \rightarrow 1$, induced by $S^\mb{Q}$, is a virtually $\tau$-group with $\mb{Q}$-Mal'cev completion $S^{\mb{Q}}$ (and the inclusion maps $N_1\rightarrow N^\mb{Q}$, $G_1\rightarrow G^\mb{Q}$). Let $\mc{H}(G^\mb{Q})$ be the set of all those virtually $\tau$-groups $S_1$ which are obtained in this way. Note that every virtually $\tau$-group with $\mb{Q}$-Mal'cev completion $S^\mb{Q}$ is isomorphic to a virtually $\tau$-group in $\mc{H}(G^\mb{Q})$. Therefore, Theorem 2 will be a consequence of the following proposition.

 \begin{proposition}
 If $S_1,S_2\in\mc{H}(G^\mb{Q})$ then $\zeta_{S_1,K}^*(s)$ and $\zeta_{S_2,K}^*(s)$ have the same abscissa of convergence for any $K\in\mathscr{F}_F^*$.
 \end{proposition}
 \begin{proof}
 If $S_1,S_2\in\mc{H}(G^\mb{Q})$, then $G_1$ and $G_2$ are finitely generated and so is the subgroup $G_1\vee G_1$ of $G^\mb{Q}$ generated by them. Then $N^\mb{Q}\cap (G_1\vee G_2)$ is finitely generated because it has finite index in $G_1\vee G_2$ and thus it is a $\tau$-group with $\mb{Q}$-Mal'cev completion $N^\mb{Q}$. It follows that $S_1\vee S_2: 1\rightarrow N^\mb{Q}\cap (G_1\vee G_2) \rightarrow G_1\vee G_2\rightarrow F\rightarrow 1$ lies in $\mc{H}(G^\mb{Q})$.  Thus, after replacing $S_1$ by $S_1\vee S_2$, we can assume that $G_2\leq G_1$. 
 
 According to Corollary \ref{Condition for having the same abscissa of convergence} and Proposition \ref{expression as an Euler product of the relative local zeta functions}, it is enough to prove that $\zeta_{{S_1}_p,K}^*(s)$ and $\zeta_{{S_2}_p,K}^*(s)$ coincide for all but a finite number of primes $p$. Choose a prime $p$ which does not divide $[G_1:G_2]=[N_1:N_2]$ and consider the morphism $(\alpha,\beta,id_F):S_2\rightarrow S_1$ given by inclusions $\alpha:N_2\rightarrow N_1$, $\beta: G_2\rightarrow G_1$. Let $(\alpha_p,\beta_p,id_F):{S_2}_p\rightarrow {S_1}_p$ be the associated morphism (Corollary \ref{functoriallity property of the Malcev completion, corollary}). Since $p$ does not divide $[N_1:N_2]$, $\alpha_p$ must be the identity, and therefore $\beta_p$ must be an isomorphism. Then ${S_1}_p={S_2}_p$ and, therefore,  $\zeta_{{S_1}_p,K}^*(s)=\zeta_{{S_2}_p,K}^*(s)$.
  \end{proof}

\subsection{Bound for the abscissa of convergence.}
It is proved in \cite[Proposition 5.6.4]{LS} (in a more general context) that if $G$ is a group containing a normal subgroup $N$ of finite index which is nilpotent and finitely generated, then $\alpha^\s(G)\leq \alpha^\s(N)+1$. In the next proposition we observe that the proof of this fact in \emph{loc.\ cit.\ }can be adapted in order to prove also that $\alpha^\lhd(G)\leq\alpha^\lhd(N)+1$. With this, we will be able to give a complete proof of Theorem 3.
\begin{proposition}\label{normal subgroup growth}
Let $S:\SES$ be a virtually nilpotent group and let $*\in\{\leq,\lhd\}$. Then for each $K\in\mathscr{F}^*_F$ we have $\alpha_{S,K}^*\leq \alpha_N^*+1$, where $\alpha_{S,K}^*$ is the abscissa of convergence of $\zeta_{S,K}^*(s)$.
\end{proposition}
\begin{proof}
Let $*\in\{\leq,\lhd\}$. There is no loss of generality if we assume that $K=F$. Let $a_n^*(S)$ be the coefficient of $n^{-s}$ in the series $\zeta_{S,F}^*(s)$, that is, $a_n^*(S)$ is the number of subgroups $A$ of $G$ such that $A*G$, $A$ has index $n$ in $G$ and $AN=G$.
It is easy to see that $a_n^*(S)\leq a_n^*(N)\psi(n)$, where $\psi(n)$ is the maximum of $\opn{der}(G/N,T)$, where $T$ ranges over all sections $C/D$ of $N$ such that $D*N$ and $[N:D]=n$ (here $\opn{der}(G/N,T)$ denotes the supremum of $|\opn{Der}_\rho(G/N,T)|$ for all actions $\rho$ of $G/N$ on $T$, and $\opn{Der}_\rho(G/N,T)$ is the set of derivations from $G/N$ into the $G/N$-group $T$ with action $\rho$). In fact, this follows because if $A$ is a subgroup of $G$ such that $AN=G$, then $B=A\cap N$ is a subgroup of $G$ such that $[N:B]=[G:A]$, $B$ is normal in $N$ if $A$ is normal in $G$, $A/B$ is a complement of $N_N(B)/B$ in $N_G(B)/B=N_N(B)A/B$, and $N_G(B)/N_N(B)\cong G/N$.  Now, it follows from \cite[Lemma 5.6.3]{LS} that $\psi(n)\leq cn$ for some constant $c$ which is independent of $n$ and therefore $a_n^*(S)\leq ca_n^*(N)n$ for all $n\in N$. Then for $t\in\mathbb{R}$ we have
\begin{align*}
\zeta_{S,F}^*(t)=\sum_{n=1}^\infty \frac{a_n^*(S)}{n^t}\leq \sum_{n=1}^\infty c\frac{a_n^*(N)n}{n^t}=c\zeta_N^*(t-1),
\end{align*}
meaning that $\zeta_{S,F}^*(t)$ converges whenever $\zeta_N^*(t-1)$ does. This concludes the proof of the proposition.
\end{proof}

\emph{Proof of Theorem 3:} I.\ $\alpha_G^*$ is the maximum of the set $\{\alpha^*_{S,K}:K\in\mathscr{F}^*_F\}$ and since each $\zeta_{S,K}^*(s)$, being an Euler product of cone integrals over $\mb{Q}$ (Theorem 1), has rational abscissa of convergence, it follows that $\alpha^*_G$ is rational. It also follows from Proposition \ref{normal subgroup growth} that $\alpha^*_G\leq \alpha^*_N+1$. Since $\alpha^*_N\leq h(N)$ \cite[Proposition 1]{GSS} and $h(N)=h(G)$, we obtain  $\alpha^*_G\leq h(G)+1$. Now II  follows from III  and the first part of III follows from Theorem 1 and the properties of cone integrals stated in the introduction. The final part of III is a direct consequence of Theorem 2.
 
\bibliographystyle{plain}


\end{document}